\documentclass[a4paper,10pt]{article}
\usepackage{amsmath,amssymb}
\usepackage[all]{xy}
\usepackage[latin1]{inputenc}
\usepackage[dvips]{graphics}
\usepackage[dvips]{epsfig}

\newcommand{\QQ}{\mathbb{Q}}

\newcommand{\arb}[1]{\includegraphics[height=5mm]{a#1.eps}}

\newcommand{\su}[1]{\widetilde{#1}}

\newcommand{\prelie}{\operatorname{PreLie}}
\newcommand{\PL}{\mathsf{PL}}
\newcommand{\PLh}{\widehat{\PL}}
\newcommand{\UPL}{{U}(\PL)}
\newcommand{\UPLh}{\widehat{U}(\PL)}

\newcommand{\pl}{\curvearrowleft}

\newcommand{\pun}{\arb{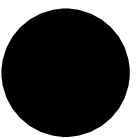}}

\newcommand{\corol}{\mathtt{Crl}}
\newcommand{\linear}{\mathtt{Lnr}}
\newcommand{\fork}{\mathtt{Frk}}

\newcommand{\maj}{\operatorname{maj}}

\newcommand{\dend}{\operatorname{Dend}}
\newcommand{\dendh}{\widehat{\operatorname{Dend}}}

\newcommand{\dun}{\includegraphics[height=3mm]{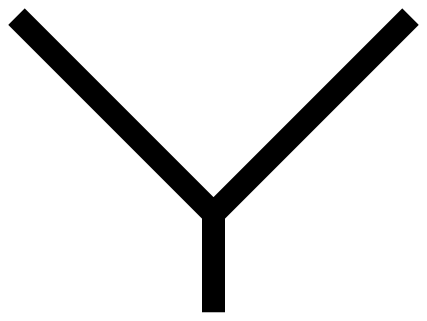}}

\newcommand{\Id}{\operatorname{Id}}
\newcommand{\aut}{\operatorname{aut}}
\newcommand{\nsym}{\textbf{Sym}}
\newcommand{\dessus}[2]{\genfrac{}{}{0mm}{1}{#1}{#2}}
\newcommand{\qbinom}[2]{\genfrac{[}{]}{0mm}{1}{#1}{#2}_q}
\renewcommand{\phi}{\varphi}

\newtheorem{theorem}{Theorem}[section] 
\newtheorem{proposition}[theorem]{Proposition}

\newtheorem{lemma}[theorem]{Lemma}

\newenvironment{proof}{\begin{trivlist}\item{\bf{Proof.}}}
  {\hfill\rule{2mm}{2mm}\end{trivlist}}

\title{A rooted-trees $q$-series lifting a one-parameter family of Lie
  idempotents} 
\author{F. Chapoton} \date{\today}

\begin{document}

\maketitle
\begin{abstract}
  We define and study a series indexed by rooted trees and with
  coefficients in $\QQ(q)$. We show that it is related to a family of
  Lie idempotents. We prove that this series is a $q$-deformation of a
  more classical series and that some of its coefficients are Carlitz
  $q$-Bernoulli numbers.
\end{abstract}

\section{Introduction}

The aim of this article is to introduce and study a series $\Omega_q$
indexed by rooted trees, with coefficients that are rational functions
of the indeterminate $q$.

The series $\Omega_q$ is in fact an element of the group $G_\PL$ of
formal power series indexed by rooted trees, which is associated to
the $\prelie$ operad by a general functorial construction of a group
from an operad \cite{Chap_2002,vdlaan,Chap_SLC,ChapLive2}. As there is
an injective morphism of operads from the $\prelie$ operad to the
dendriform operad, there is an injection of groups from $G_{\PL}$ to
the group $G_{\dend}$, which is a group of formal power series indexed
by planar binary trees. This means that each series indexed by rooted
trees can be mapped to a series indexed by planar binary trees, in a
non-trivial way.

There is a conjectural description of the image of this injection of
groups (see \cite[Corollary 5.4]{mould}). This can be stated roughly
as the intersection in a bigger space (spanned by permutations) of the
dendriform elements with the Lie elements. The inclusion of the image
in the intersection is known, but the converse is not.

One starting point of this article was the existence of a
one-parameter family of Lie idempotents belonging to the descent
algebras of the symmetric groups \cite{Duchamp_et_al_1994,ncsf2}. As
Lie idempotents, these are in particular Lie elements. As elements of
the descent algebras, these are also dendriform elements. Therefore,
according to the conjecture stated above, they should belong to the
image of $G_{\PL}$ in $G_{\dend}$.

Bypassing the conjecture, we prove this by exhibiting an element
$\Omega_q$ of $G_\PL$ and then showing that its image is the expected
sum of Lie idempotents.

We then obtain several results on $\Omega_q$. First, we prove that the
series $\Omega_q$ has only simple poles at non-trivial roots of unity
and in particular, can be evaluated at $q=1$. Then we show that
$\Omega_q$ is a $q$-deformation of a classical series $\Omega$ which
is its value at $q=1$. We also compute the value at $q=0$ and the
appropriate limit value when $q=\infty$.

We then consider the images of $\Omega_q$ in some other groups. There
are two morphisms of groups from $G_\PL$ to usual groups of formal
power series in one variable. Looking at corollas only, one gets a map
from $G_{\PL}$ to the group of formal power series with constant term
$1$ for multiplication. The image of $\Omega_q$ is then the generating
function of the $q$-Bernoulli numbers introduced by Carlitz, that
appear quite naturally here.

On the other hand, looking at linear trees only, one gets a map from
$G_{\PL}$ to the composition group of formal power series without
constant term. The image of $\Omega_q$ is then a $q$-logarithm.

\medskip

The present work received support from the ANR grant
BLAN06-1\underline{ }136174.

Many useful computations and checks have been done using MuPAD.

\section{General setting}

We will work over the field $\QQ$ of rational numbers and over the field
$\QQ(q)$ of fractions in the indeterminate $q$.

We have tried to avoid using operads as much as possible, but this
language is needed at some points in this article. The reader may
consult \cite{loday,Chap_SLC} as references. The symbol $\circ$ will
denote the composition in an operad and the symbols $\flat$ and
$\natural$ will serve to note the place where composition is done.

\subsection{Pre-Lie algebras}

Recall (see for instance \cite{ChapLive1}) that a \textbf{pre-Lie
  algebra} is a vector space $V$ endowed with a bilinear map $\pl$
from $V \otimes V$ to $V$ satisfying the following axiom:
\begin{equation}
  \label{axiomePL}
  (x \pl y) \pl z -x \pl (y \pl z)= (x \pl z) \pl y -x \pl (z \pl y).
\end{equation}
This is sometimes called a right pre-Lie algebra.

The pre-Lie product $\pl$ defines a Lie bracket on $V$ as follows:
\begin{equation}
  [x,y]=x\pl y-y\pl x.
\end{equation}
One can easily check that the pre-Lie axiom (\ref{axiomePL}) implies
the Jacobi identity for the anti-symmetric bracket $[\,,\,]$.

The pre-Lie product $\pl$ can also be considered as a right action
$\pl$ of the associated Lie algebra $(V,[\,,\,])$ on the vector space
$V$. Indeed, one has
\begin{equation}
  (x \pl y) \pl z -(x \pl z) \pl y = x \pl [y , z].
\end{equation}
This should not be confused with the adjoint action of a Lie algebra
on itself.

\subsection{Free pre-Lie algebras}

The free pre-Lie algebras have a simple description using rooted
trees. Let us recall briefly this description and other properties.
Details can be found in \cite{ChapLive1}.

A \textbf{rooted tree} is a finite, connected and simply connected
graph, together with a distinguished vertex called the root. We will
picture rooted trees with their root at the bottom and orient
(implicitly) the edges towards the root. There are two distinguished
kinds of rooted trees: corollas (every vertex other than the root is
linked to the root by an edge) and linear trees (at every vertex,
there is at most one incoming edge), see Fig. \ref{fig:trees}. A
\textbf{forest of rooted trees} is a finite graph whose connected
components are rooted trees.

The free pre-Lie algebra $\PL(S)$ on a set $S$ has a basis indexed by
rooted trees decorated by $S$, \textit{i.e.} rooted trees together
with a map from their set of vertices to $S$.

The pre-Lie product $T \pl T'$ of a tree $T'$ on another one $T$ is
given by the sum of all possible trees obtained from the disjoint
union of $T$ and $T'$ by adding an edge from the root of $T'$ to one
of the vertices of $T$ (the root of the resulting tree is the root of
$T$).

In particular, we will denote by $\PL$ the free pre-Lie algebra on one
generator. This is the graded vector space $\PL=\oplus_{n\geq
  1}{\PL_n}$ spanned by unlabeled rooted trees, where the degree of a
tree $T$ is the number $\#T$ of its vertices. The pre-Lie product
obviously preserves this grading. We will denote by $\star$ the
associative product in the universal enveloping algebra $\UPL$ of the
Lie algebra $\PL$.

There exists a unique isomorphism $\psi$ of graded right $\PL$-modules
between the free right $\UPL$-module on one generator $g$ of degree
$1$ and the $\PL$-module $(\PL, \pl)$ such that $\psi$ maps the
generator $g$ to $\pun$, the unique rooted tree with one vertex.

This means that there is a commutative diagram as follows:
\begin{equation}
  \label{action_diagramme}
\xymatrix{
\QQ g \ar[d]^{\Id \otimes \star} \otimes \UPL \otimes \PL
\ar[r]^{\quad \, \psi
  \otimes \Id} & \PL \otimes \PL \ar[d]^{\pl}\\
\QQ g        \otimes \UPL  \ar[r]^{\psi} & \PL \\
}
.
\end{equation}

One can use the bijection $\psi$ and the canonical basis of $\PL$ to
get a canonical basis of the enveloping algebra $\UPL$ indexed by
forests of rooted trees. The degree of a forest $F$ is the number of
its vertices $\# F$. The image of the usual inclusion of $\PL$ in
$\UPL$ is the subspace spanned by rooted trees. In this basis of
$\UPL$, there is a nice combinatorial description of the associative
product $\star$. Let $F$ and $F'$ be forests in $\UPL$.  The product
$F \star F'$ is the sum of all possible forests, obtained from the
disjoint union of $F$ and $F'$ by the addition of some edges (possibly
none), each of these new edges going from some root of $F'$ to some
vertex of $F$.

There is a canonical projection $\pi$ from $\UPL$ to $\PL$, defined
using the canonical basis of $\UPL$ by projection on the subspace
spanned by rooted trees, annihilating the empty forest and all
forests that are not trees.

\begin{lemma}
  \label{pi_exp}
  Let $F$ be a forest in $\UPL$ and $T$ be a rooted tree in $\PL$.
  Then one has $\pi(F \star T)=\pi(F) \pl T$.
\end{lemma}
\begin{proof}
  If $F$ is not a tree, then each term of $F \star T$ is not a tree,
  therefore both sides vanish. If $F=\pi(F)$ is a tree, then $F \star
  T$ is the sum of $\pi(F) \pl T$ with the disjoint union of $F$ and
  $T$. Therefore $\pi(F \star T)=\pi(F) \pl T$.
\end{proof}

\begin{lemma}
  \label{injection}
  For all $n\geq 1$, the maps $T \mapsto \pun \pl T$ and $T \mapsto T
  \pl \pun$ are injective from $\PL_n$ to $\PL_{n+1}$.
\end{lemma}
\begin{proof}
  This is obvious for the first map, which is even an injection on the
  set of rooted trees. For the second map, this follows from the fact
  that enveloping algebras are integral domains, by restriction of
  the commutative diagram (\ref{action_diagramme}).
\end{proof}

In the sequel, we will always work in the completed vector space
$\PLh=\prod_{n\geq 1}{\PL_n}$ and with its completed enveloping
algebra $\UPLh$. All the results above are still true in this setting.

There is a group associated to each operad, see
\cite{Chap_2002,vdlaan,Chap_SLC,ChapLive2}. We will need the group
$G_\PL$ associated to the $\prelie$ operad. Its elements are the
elements of $\PLh$ whose homogeneous component of degree $1$ is
$\pun$. Product in $G_\PL$ is defined using the composition of the
$\prelie$ operad and $\pun$ is the unit in $G_\PL$. This group is
contained in the bigger monoid $\PLh$, on which it therefore acts on
the right and on the left. The right action respects all the
operations on $\PLh$ induced by the product $\pl$, including the
product and the action of $\UPLh$.

Let us now introduce a special element of $G_{\PL}$, for later use.
Let $\exp^* \in G_{\PL}$ be
\begin{equation}
  \label{defexp}
  \exp^* = \pun \pl \left((\exp(\pun)-1) / \pun\right).
\end{equation}
The series $\exp^*$ is very classical, and its coefficients
are known as the Connes-Moscovici coefficients (see \cite{Chap_2002}).

Let us consider the left action of $\exp^*$ on $\PLh$. Let $T$ be an
element of $\PLh$. Then $\exp^*(T)$ of $\PLh$ is defined by
\begin{equation}
  \exp^*(T)=\sum_{n \geq 1} \frac{1}{n!} ((T \pl T) \pl \dots ) \pl T,
\end{equation}
where there are $n$ copies of $T$ in the $n^{th}$ term. As $\exp^*$
belongs to the group $G_\PL$, the map $\exp^*$ defines a bijection
from $\PLh$ to itself.

Let us now relate the usual exponential map $\exp$ to the map
$\exp^*$.

Let $T$ be an element of $\PLh$. Let $\exp(T)$ be the exponential of
$T$ in $\UPLh$ (which is defined by the usual series and using the
$\star$ product). The map $\exp$ defines a bijection from $\PLh$ to
the set of group-like elements of $\UPLh$.

Therefore, the composite map $\exp^* \circ \exp^{-1}$ is a bijection
from the set of group-like elements in $\UPLh$ to $\PLh$. Let us show
that this composite map is just a restriction of the canonical
projection $\pi$.

\begin{proposition}
  \label{2groupes}
  Let $T$ be an element of $\PLh$. One has $\pi(\exp(T))=\exp^*(T)$.
\end{proposition}

\begin{proof}
  Let $F$ be in $\UPLh$. From Lemma \ref{pi_exp} above, one knows that
  $\pi(F \star T)$ is exactly $\pi(F) \pl T$. This implies that
  \begin{equation}
    \pi(T^{\star n})=((T \pl T) \dots )\pl T,
  \end{equation}
  for all $n \geq 1$, hence the result.
\end{proof}

\section{The classical case}

Let us start by recalling the definition of a classical element
$\Omega$ of $\PLh$ with rational coefficients. It was considered under
the name of $\log^*$ in \cite{Chap_2002} and has been since studied in
\cite{Murua,Wright_2003,ebrahimi-fard08,calaque-2008}.

\begin{proposition}
  There is a unique solution $\Omega$ in $\PLh_\QQ$ to the equation
  \begin{equation}
    \label{class_eq1}
    \pun \pl\left(\frac{\Omega}{\exp(\Omega)-1}\right) = \Omega,
  \end{equation}
  where $\frac{\Omega}{\exp(\Omega)-1}$ is in the completed enveloping
  algebra $\UPLh$.
\end{proposition}

\begin{proof}
  Let us write $\Omega=\sum_{n\geq 1} \Omega_n$ where each $\Omega_n$
  is homogeneous of degree $n$. 

  Recall the Taylor expansion
  \begin{equation}
    \frac{x}{\exp(x)-1}=\sum_{k\geq 0} \frac{B_k}{k!} x^k,
  \end{equation}
  where the $B_k$ are the Bernoulli numbers.

  Then the homogeneous component of
  degree $n$ of
  equation (\ref{class_eq1}) is
  \begin{equation}
    \label{homog_class_eq1}
    \Omega_n= \sum_{k \geq 0} \frac{B_k}{k!} \sum_{\dessus{m_1\geq
      1,\dots,m_k\geq 1}{m_1+\dots+m_k=n-1}}
    ((\pun \pl \Omega_{m_k} ) \dots ) \pl \Omega_{m_1}.
  \end{equation}

  This gives a recursive definition of $\Omega_n$, which implies the
  existence and uniqueness of $\Omega$.
\end{proof}

Remark: one can use equation (\ref{homog_class_eq1}) to compute
$\Omega$ up to order $n$ in a $O(n^3)$ number of pre-Lie operations.

As the element $\frac{\Omega}{\exp(\Omega)-1}$ is invertible in the
completed enveloping algebra, equation (\ref{class_eq1}) is also equivalent to
the following equation:
\begin{equation}
  \label{class_eq2}
  \Omega \pl \left(\frac{\exp(\Omega)-1}{\Omega}\right) =\pun.
\end{equation}

One can interpret equation (\ref{class_eq2}) as follows.
\begin{proposition}
  \label{inverse_omega}
  The series $\Omega$ is the inverse of $\exp^*$ in the group $G_\PL$.
\end{proposition}
\begin{proof}
  By right action by the inverse $\Omega^{-1}$ of $\Omega$ in $G_\PL$
  on (\ref{class_eq2}), one shows that $\Omega^{-1}$ satisfies the
  same equation (\ref{defexp}) as $\exp^*$.
\end{proof}

There is another equation for $\Omega$.

\begin{proposition}
  \label{uniqueness2}
  The series $\Omega$ is the unique non-zero solution in
  $\PLh_\QQ$ to the equation
  \begin{equation}
    \label{class_eq3}
    \Omega \pl (\exp(\Omega)-1)  =  \pun\pl\Omega,
  \end{equation}
  where $\exp(\Omega)-1$ is in the completed enveloping algebra $\UPLh$.
\end{proposition}

\begin{proof}
  First, by right action on (\ref{class_eq2}) by $\Omega$, one can see
  that the unique solution $\Omega$ of (\ref{class_eq1}) is indeed a
  solution of (\ref{class_eq3}).

  Let us now prove uniqueness of a non-zero solution. Let $\Omega$ be
  any solution of (\ref{class_eq3}). Let us write $\Omega=\sum_{n\geq
    1} \Omega_n$ where each $\Omega_n$ is homogeneous of degree $n$.

  Then the homogeneous component of degree $n$ of equation
  (\ref{class_eq3}) is
  \begin{equation}
    \label{homog_class_eq3}
    \pun \pl \Omega_{n-1}=\sum_{k \geq 1} \frac{1}{k!} \sum_{\dessus{m_1\geq
      1,\dots,m_k\geq 1,\ell \geq 1}{m_1+\dots+m_k+\ell=n}}
    ((\Omega_{\ell} \pl \Omega_{m_{k}} ) \dots ) \pl \Omega_{m_1}.
  \end{equation}

  If $n=2$, this implies that $\Omega_1$ is either $0$ or $\pun$.

  Assume now that $\Omega$ is not zero. Let $d$ be the degree of the
  first non-zero homogeneous component $\Omega_d$ of $\Omega$. Assume
  that $d>1$. Then the equation (\ref{homog_class_eq3}) in degree
  $d+1$, together with Lemma \ref{injection}, gives that $\Omega_d=0$, a
  contradiction. Therefore necessarily, one has $d=1$ and
  $\Omega_1=\pun$.
 
  Let us look at the homogeneous component (\ref{homog_class_eq3}) in
  degree $n+1\geq 2$. The only terms involving $\Omega_n$ are $\pun
  \pl \Omega_n $ in the left-hand side and $\Omega_1 \pl \Omega_n $,
  $\Omega_n \pl \Omega_1 $ in the right hand-side. As
  $\Omega_1=\pun$, two of them cancel out and one gets a recursive
  expression of $ \Omega_n \pl \pun$ in terms of some $\Omega_j$ for
  $j<n$.

  Using Lemma \ref{injection}, this provides a recursive description
  of $\Omega$ (that may or may not possess a solution) and proves its
  uniqueness.
\end{proof}

The exponential of $\Omega$ has a simple shape.

\begin{proposition}
  \label{prop_expomega}
  In the enveloping algebra $\UPLh$, one has
  \begin{equation}
    \label{expomega}
    \exp(\Omega)=\sum_{n \geq 0} \frac{1}{n!} {\pun \, \pun \dots \pun},
  \end{equation}
  where, in the $n^{th}$ term, the forest has $n$ nodes.
\end{proposition}

\begin{proof}
  This is an equation for the exponential $\exp(\Omega)$ of the
  element $\Omega$ in the Lie algebra $\PLh$. By Proposition
  \ref{2groupes}, it is enough to prove that
  \begin{equation}
    \exp^* (\Omega)=\pun,
  \end{equation}
  because the image by $\pi$ of the right-hand side of
  (\ref{expomega}) is $\pun$.

  But this amounts to say that $\exp^*$ is the inverse of $\Omega$ in
  the group $G_\PL$. This is nothing else that Proposition \ref{inverse_omega}.
\end{proof}

It follows that
\begin{equation}
  \label{expo_coro}
  \Omega \pl (\exp(\Omega)-1)
  = \sum_{n \geq 2}  \frac{1}{(n-1)!} \corol^\natural_{n} \circ_\natural \Omega,
\end{equation}
where $\corol^\natural_n$ is the corolla with $n-1$ leaves and with
root labeled by $\natural$, see Fig. \ref{fig:trees}.

\begin{proposition}
  The series $\Omega$ is the unique non-zero solution in
  $\PLh_\QQ$ to the equation
  \begin{equation}
    \sum_{n \geq 2}  \frac{1}{(n-1)!} \corol^\natural_{n} \circ_\natural \Omega  
    =  \pun\pl\Omega,
  \end{equation}
  where $\sum_{n \geq 2} \frac{1}{(n-1)!}\corol^\natural_{n}$ is in $\PLh$.
\end{proposition}

\begin{proof}
  This follows from Eq. (\ref{expo_coro}) and Prop. \ref{uniqueness2}.
\end{proof}

\section{The quantum case}

We will introduce now an element $\Omega_q$ in $\PLh$ with coefficients
in $\QQ(q)$. We will show later that this is a $q$-deformation of $\Omega$.

If $A=\sum_{n\geq 1} A_n$ is an element of $\PLh$, let $A[q]$ be the
$q$-shift of $A$ defined by
\begin{equation}
  A[q]=\sum_{n\geq 1} q^n A_n.
\end{equation}

\begin{proposition}
  \label{def_omegaq}
  There exists a unique solution $\Omega_q$ in $\PL_{\QQ(q)}$ to the equation
  \begin{equation}
    \label{quant_eq}
    \Omega_q[q]  \pl (\exp(\Omega)-1) +  \Omega_q[q]-\Omega_q = \pun\pl \Omega_q
    +(q-1)\,  \pun.
  \end{equation}
  Moreover, the series $\Omega_q$ has coefficients in the ring of
  fractions with poles only at roots of unity.
\end{proposition}

\begin{proof}
  Let us write $\Omega_q=\sum_{n\geq 1} \Omega_{q,n}$ where each
  $\Omega_{q,n}$ is homogeneous of degree $n$. The homogeneous
  component of degree $1$ of (\ref{quant_eq}) implies that
  $\Omega_{q,1}=\pun$.

  Then for $n\geq 2$, the homogeneous component of degree $n$ of
  equation (\ref{quant_eq}) is
  \begin{equation}
       \label{quant_recursion}
       (q^n -1) \Omega_{q,n}= \pun \pl \Omega_{q,n-1} - \sum_{k \geq 1} \frac{1}{k!} \sum_{\dessus{m_1\geq
      1,\dots,m_k\geq 1,\ell \geq 1}{m_1+\dots+m_k+\ell=n}} q^\ell
    ((\Omega_{q,\ell} \pl \Omega_{m_{k}} ) \dots ) \pl \Omega_{m_1}.
  \end{equation}

  This provides an explicit recursion for $\Omega_{q,n}$ in terms of
  $\Omega_{q,j}$ and $\Omega_j$ for $j<n$. This gives existence and
  uniqueness and also implies that $\Omega_q$ has coefficients with
  poles only at roots of unity.
\end{proof}

On can reformulate the equation for $\Omega_q$.

\begin{proposition}
  The series $\Omega_q$ is the unique solution in
  $\PLh_{\QQ(q)}$ to the equation
  \begin{equation}
    \label{quant_eq2}
    \sum_{n \geq 1}  \frac{1}{(n-1)!}\corol^\natural_{n} \circ_\natural \Omega_q[q] -\Omega_q = \pun\pl \Omega_q
    +(q-1)\,  \pun.
  \end{equation}
\end{proposition}

\begin{proof}
  This follows from Eq. (\ref{expo_coro}) and Prop. \ref{def_omegaq}.
\end{proof}

Let $\fork^\natural_{\ell,n}$ be the rooted tree with a linear trunk
of $\ell$ vertices, a vertex $\natural$ on top of this trunk and a
corolla with $n$ leaves on top of the vertex $\natural$, see Fig.
\ref{fig:trees}. We will call this a \textbf{fork}. One has
$\fork_{\ell,n}=\linear_{\ell+1}^\flat \circ_\flat
\corol^\natural_{n+1}$.

\begin{figure}
  \begin{center}
    \scalebox{0.4}{\input{corolle_etc.pstex_t}}
    \caption{Rooted trees: $\linear^{\flat}_5$, $\corol^\natural_6$
      and $\fork^{\natural}_{4,5}=\linear^{\flat}_5 \circ_\flat\corol^\natural_6$.}
    \label{fig:trees}
  \end{center}
\end{figure}
	
\begin{proposition}
  The series $\Omega_q$ is the unique solution in $\PLh_{\QQ(q)}$ to
  the equation
  \begin{equation}
    \label{fork_eq}
    \Omega_q=\sum_{\ell\geq 0}\sum_{n \geq 0} \frac{(-1)^\ell}{n!}
    \fork^\natural_{\ell,n} \circ_\natural \Omega_q[q] + 
    (1-q)\sum_{\ell \geq 1} (-1)^{\ell-1} \linear_{\ell}.
  \end{equation}
\end{proposition}

\begin{proof}
  Let us compute the right-hand side of Eq. (\ref{fork_eq}), using Eq.
  (\ref{quant_eq2}) for $\Omega_q$, written as
  \begin{equation}
    \Omega_q  + \pun\pl \Omega_q  +(q-1)\,  \pun= \sum_{n \geq 1}  \frac{1}{(n-1)!}\corol^\natural_{n} \circ_\natural    \Omega_q[q].
  \end{equation}

  One gets
  \begin{equation}
    \sum_{\ell\geq 1} (-1)^{\ell-1}
    \linear^\flat_{\ell} \circ_\flat (\Omega_q+\pun\pl \Omega_q+(q-1)\,\pun) + 
    (1-q)\sum_{\ell \geq 1} (-1)^{\ell-1} \linear_{\ell}.    
  \end{equation}
  As $\linear^\flat_{\ell} \circ_\natural \pun=\linear_{\ell}$ and
  $\linear^\flat_\ell \circ_\flat (\pun \pl
  \Omega_q)=\linear_{\ell+1}^\flat \circ_\flat \Omega_q$, the
  two right-most terms cancels, and the sum simplifies to
  \begin{equation}
    \sum_{\ell\geq 1} (-1)^{\ell-1}
    \linear^\flat_{\ell} \circ_\flat \Omega_q
    -\sum_{\ell\geq 2} (-1)^{\ell-1}
    \linear^\flat_{\ell} \circ_\flat \Omega_q,
  \end{equation}
  which is just $\Omega_q$. This proves that $\Omega_q$ does satisfy
  Eq. (\ref{fork_eq}).

  It is then easy to see that (\ref{fork_eq}) has only one solution in
  $\PLh_{\QQ(q)}$ by rewriting it as a recursion for the homogeneous
  components $\Omega_{q,n}$.
\end{proof}

\section{Image in the free dendriform algebra}

We describe in this section the image of $\Omega_q$ by the usual
morphism from the free pre-Lie algebra to the free dendriform algebra.
We show that this image is related to a family of Lie idempotents in
the descent algebras of the symmetric groups. One deduces from that a
nice explicit formula, that will be used later to get arithmetic
information on $\Omega_q$.

\subsection{Dendriform algebra}

Recall that a \textbf{dendriform algebra} (notion due to Loday, see
\cite{loday}) is a vector space $V$ endowed with two bilinear maps
$\succ$ and $\prec$ from $V \otimes V$ to $V$ satisfying the following
axioms:
\begin{align}
  \label{axiom1}
 x \prec ( y \prec z)+x \prec  (y \succ z)  &= (x \prec  y) \prec z,\\
  \label{axiom2}
  x \succ ( y \prec z) &= (x \succ  y) \prec z,\\
  \label{axiom3}
   x \succ (y \succ z) &= (x \succ y) \succ z +(x \prec y) \succ z. 
\end{align}

Any dendriform algebra has the structure of a pre-Lie algebra given by
\begin{equation}
  x \pl y = y \succ x - x \prec y.
\end{equation}

Any dendriform algebra has the structure of an associative algebra given by
\begin{equation}
  x * y = x \succ y + x \prec y.
\end{equation}

Remark: Eq. (\ref{axiom2}) means that one can safely forget some
parentheses. Eq. (\ref{axiom1}) and (\ref{axiom3}) can be rewritten as
\begin{align}
  x \prec ( y * z) &=(x \prec  y) \prec z ,\\
   x \succ (y \succ z) &= (x * y) \succ z. 
\end{align}

Let $\dend(S)$ be the free dendriform algebra over a set $S$. This has
an explicit basis indexed by \textbf{planar binary trees} with
vertices decorated by $S$. For an example of a planar binary tree, see
Fig. \ref{descente}. In particular, the free dendriform algebra on one
generator, denoted by $\dend$, has a basis indexed by planar binary
trees. This is a graded vector space, the degree $\# t$ of a planar
binary tree $t$ being the number of its inner vertices.

There is a unique morphism $\phi$ of pre-Lie algebras from $\PL$ to
$\dend$ that maps the rooted tree $\pun$ to the planar binary tree
$\dun$. This extends uniquely to a continuous morphism $\phi$ from
$\PLh$ to the completion $\dendh$ of $\dend$.

Remark: with some care, one can add a unit $1$ to the free dendriform
algebra $\dend$. Then one has $1*x=1 \succ x =x =x \prec 1=x*1$, but
one has to pay attention to never write neither $1 \prec x$ nor $x \succ
1$. We will use this convention in the sequel.

There are two kinds of special planar binary trees: the left combs
and the right combs. They can be defined as follows. Let $L=\sum_{n
  \geq 1}L_n$ be the unique solution in $\dendh$ to the equation
\begin{equation}
  \label{defL}
  L=\dun+L \succ \dun=(1+L) \succ \dun,
\end{equation}
and let $R=\sum_{n\geq 1} R_n$ be the unique solution in $\dendh$ to
\begin{equation}
  \label{defR}
  R=\dun+\dun \prec R=\dun \prec (1+R).
\end{equation}

Then $L_n$ is called the left comb with $n$ vertices and $R_n$ be the
right comb with $n$ vertices.

If $A=\sum_{n\geq1} A_n$ is an element of $\PLh$ or $\dendh$, the
\textbf{suspension} of $A$ is $\su{A}=\sum_{n\geq1} (-1)^{n-1} A_n$.

\begin{proposition}
  \label{inverseLR}
  The inverse of $1+R$ with respect to the $*$ product is  $1-\su{L}$.
\end{proposition}

\begin{proof}
  One has $\su{L}=\dun-\su{L} \succ \dun$. Let us compute
  \begin{equation}
    (1-\su{L})*(1+{R})=1+R-\su{L}*(1+R).
  \end{equation}
  By the definition of $*$ and the convention on the unit $1$, this is
  \begin{equation}
    1+R-\su{L} \prec (1+R)-\su{L} \succ R
  \end{equation}
  By Eq. (\ref{defL}), this becomes
  \begin{equation}
    1+R-\dun \prec(1+R)+\su{L} \succ \dun \prec (1+R) -\su{L} \succ R.
  \end{equation}
  The last two terms cancel by Eq. (\ref{defR}) and one gets
  \begin{equation}
    1+R-\dun \prec(1+R),
  \end{equation}
  which is just $1$, again by Eq. (\ref{defR}).
\end{proof}

\subsection{Equation for the dendriform image of $\Omega_q$}

Let us define a series $E=\sum_{n \geq 1} n L_n$ in $\dendh$. One can
easily show that
\begin{equation}
  \label{eq_E}
  E=L+ E \succ \dun.
\end{equation}

\begin{lemma}
  \label{image_linear}
  The series $B^\flat=\phi(\sum_{n\geq 1} \linear^\flat_n)$ satisfies
  \begin{equation}
    \label{Bflat}
    B^\flat=\dun^\flat+B^\flat \succ \dun -\dun \prec B^\flat.
  \end{equation}
\end{lemma}
\begin{proof}
   This comes from a similar equation in $\PL$. Let
   $\linear^\flat=\sum_{n\geq 1} \linear^\flat_n$. Then
   \begin{equation}
     \linear^\flat=\pun{}^\flat+ \pun \pl \linear^\flat,
   \end{equation}
   as one can easily check.
\end{proof}

These relations can be taken as definitions of the elements $E$ and
$B^\flat$ of $\dendh$. One can forget the marking $\flat$ in $B^\flat$
to define a series $B$.

\begin{proposition}
  \label{powersum}
  The series $B=\phi(\sum_{n\geq 1} \linear_n)$ satisfies
  \begin{equation}
    E = (1+L) * B.
  \end{equation}
\end{proposition}
\begin{proof}
  One has to show that $E=(1+L)*B$. It is enough to prove that $(1+L)*B$
  does satisfy the defining relation (\ref{eq_E}) of $E$.

  One computes, using Eq. (\ref{Bflat}) for $B$,
  \begin{equation}
    (1+L)*B=(1+L)*(\dun+B \succ \dun - \dun \prec B).
  \end{equation}
  Expanding the $*$ product, this is
  \begin{equation}
    (1+L)\succ \dun+L \prec \dun 
    +(1+L)\succ(B \succ \dun - \dun \prec B)
    +L\prec(B \succ \dun - \dun \prec B).
  \end{equation}
  Using Eq. (\ref{defL}) and the dendriform axioms, this becomes
  \begin{equation}
    L+L\prec \dun +((1+L)*B) \succ \dun - (1+L) \succ \dun \prec B +L
    \prec (B \succ \dun - \dun \prec B).
  \end{equation}
  Using Eq. (\ref{defL}) again, one gets 
  \begin{equation}
    L+((1+L)*B) \succ \dun+L\prec (\dun-B + B \succ \dun - \dun \prec B).
  \end{equation}
  This simplifies, by Eq. (\ref{Bflat}) for $B$, to
  \begin{equation}
     L+((1+L)*B) \succ \dun,
  \end{equation}
  as expected.
\end{proof}

\begin{lemma}
  \label{image_corol}
  The image of $\sum_{n\geq 1} \frac{1}{(n-1)!}\corol^\natural_n$ by
  $\phi$ is
  \begin{equation}
    (1+R) \succ \dun^\natural \prec (1-\su{L}).
  \end{equation}
\end{lemma}
\begin{proof}
  This was proved in \cite{ronco1,ronco2,Chap_CMMQ}.
\end{proof}

\begin{proposition}
  \label{image_fork}
  The image of $\sum_{\ell\geq 0} \sum_{n\geq 0}
  \frac{(-1)^\ell}{n!}\fork^\natural_{\ell,n}$ by $\phi$ is
  \begin{equation}
    (1+R) * \dun^\natural * (1-\su{L}),
  \end{equation}
  where $\dun^\natural$ is the planar binary tree $\dun$ with vertex
  labeled by $\natural$.
\end{proposition}
\begin{proof}
  Let $D=(1+R) * \dun^\natural * (1-\su{L})$. Let us first show that
  \begin{equation}
    \label{defD}
    D=(1+R) \succ \dun^\natural \prec (1-\su{L}) + \dun \prec D - D
    \succ \dun. 
  \end{equation}
  Expanding the $*$ product, one computes
  \begin{equation}
    D=((1+R)* \dun^\natural) \prec (1-\su{L}) - ((1+R)* \dun^\natural) \succ \su{L}.
  \end{equation}
  Then one gets, by expanding again, 
  \begin{equation}
    (1+R) \succ \dun^\natural \prec (1-\su{L}) + (R \prec \dun^\natural) \prec (1-\su{L}) - ((1+R)* \dun^\natural) \succ \su{L}.
  \end{equation}
  Using the dendriform axioms, this is
  \begin{equation}
    (1+R) \succ \dun^\natural \prec (1-\su{L}) + R \prec
    (\dun^\natural * (1-\su{L})) - ((1+R)* \dun^\natural) \succ \su{L}.
  \end{equation}
  Then by Eq. (\ref{defL}) and (\ref{defR}), this can be rewritten
  \begin{multline}
    (1+R) \succ \dun^\natural \prec (1-\su{L}) + (\dun \prec (1+R)) \prec
    (\dun^\natural * (1-\su{L})) \\
    - ((1+R)* \dun^\natural) \succ ((1-\su{L}) \succ \dun).    
  \end{multline}
  One gets, using the dendriform axioms,
  \begin{equation}
     (1+R) \succ \dun^\natural \prec (1-\su{L}) + \dun \prec ((1+R)*
    \dun^\natural * (1-\su{L})) - ((1+R)* \dun^\natural *
    (1-\su{L})) \succ \dun.       
  \end{equation}
  This proves the equation (\ref{defD}) for $D$.

  Let us show now that $D'=(\sum_{\ell\geq 1}(-1)^{\ell-1}\phi(\linear_\ell^\flat))
  \circ_\flat (\sum_{n\geq 1}\frac{1}{(n-1)!}\phi(\corol_n^\natural))$ does
  satisfy the same equation as $D$.

  By Lemma \ref{image_corol}, one has
  \begin{equation}
    D'=\su{B}^\flat \circ\flat ((1+R) \succ \dun^\natural \prec (1-\su{L})).
  \end{equation}
  By Lemma \ref{image_linear}, one has
  $\su{B}^\flat=\dun^\flat-\su{B}^\flat \succ \dun +\dun \prec
  \su{B}^\flat$, hence
  \begin{equation}
    D'=(1+R) \succ \dun^\natural \prec (1-\su{L})+ \dun \prec D' - D'
    \succ \dun. 
  \end{equation}

  By uniqueness of the solution $D$ of Eq. (\ref{defD}), one has
  $D=D'$, \textit{i.e.}
  \begin{equation}
    (1+R) * \dun^\natural * (1-\su{L}) = (\sum_{\ell\geq
      1}(-1)^{\ell-1}\phi(\linear_\ell^\flat)) \circ_\flat (\sum_{n\geq 1}\frac{1}{(n-1)!}\phi(\corol_n^\natural)).
  \end{equation}
  Therefore 
  \begin{equation}
    (1+R) * \dun^\natural * (1-\su{L})=\phi\left((\sum_{\ell\geq
      1}(-1)^{\ell-1}\linear_{\ell}^\flat) \circ_\flat (\sum_{n\geq 1}\frac{1}{(n-1)!}\corol_n^\natural)\right),
  \end{equation}
  which is exactly the expected image by $\phi$ of a sum over forks.
\end{proof}

One can now deduce a useful functional equation for the image of
$\Omega_q$ by $\phi$, using only the associative product $*$ of
$\dend$.

\begin{proposition}
  \label{eq_phi_omega}
  The series $\phi(\Omega_q)$ is the unique solution in $\dendh$ of
  \begin{equation}
    \phi(\Omega_q)=(1-\su{L})^{-1} * \phi(\Omega_q)[q] * (1-\su{L}) + (1-q)
     \sum_{\ell \geq 1} (-1)^{\ell-1} \phi(\linear_\ell).
  \end{equation}
\end{proposition}

\begin{proof}
  Let us start from Eq. (\ref{fork_eq}). By Prop. \ref{image_fork}, we
  know the image by $\phi$ of the sum over forks. One gets
  \begin{equation}
    \phi(\Omega_q)=((1+R) * \dun^\natural * (1-\su{L})) \circ_\natural \phi(\Omega_q)[q] + 
    (1-q)\sum_{\ell \geq 1} (-1)^{\ell-1} \phi(\linear_{\ell}).  
  \end{equation}
  Then one can used Prop. \ref{inverseLR} to replace $1+R$ by the
  inverse of $1-\su{L}$.
\end{proof}

\subsection{Explicit formula}

We will prove in this section that the image of $\Omega_q$ by $\phi$
coincides (in some sense) with a known family of Lie idempotents, and
has an explicit description using $q$-binomial coefficients, descents
and major indices of planar binary trees. To obtain this description,
we use a result on noncommutative symmetric functions. We refer to the
articles \cite{ncsf1,ncsf2,ncsf3} for background on this subject. We
will use the notations of \cite{ncsf2}.

The algebra $\nsym$ of noncommutative symmetric function is the free
unital associative algebra on generators $S_1,S_2,\dots$. It is a
graded algebra (with $S_i$ of degree $i$), with a basis $(S_I)_I$
indexed by compositions. There is another basis $(R_I)_I$ obtained from
the basis $(S_I)_I$ by M{\"o}bius inversion on compositions ordered by
refinement. By convention, $S_0$ is the unit of $\nsym$.

As $\nsym$ is free, there is a unique morphism $\theta$ from $\nsym$
to $\dend$ which maps $S_i$ to the left comb $L_i$ for each $i\geq 0$,
with the convention that $L_0$ is the unit of $\dend$. 

One can check that $\theta$ is the usual morphism from $\nsym$ to
$\dend$, considered for instance in \cite[\S 4.8]{pbt} and \cite{lodayronco}.

In $\nsym$, there are elements $\Psi_i$ for $i\geq 1$, uniquely
defined by the conditions
\begin{equation}
  n S_n = \sum_{i=0}^{n-1} S_i * \Psi_{n-i},
\end{equation}
for all $n\geq 1$.

\begin{proposition}
  \label{theta_psi}
  The image of $\Psi_i$ by $\theta$ is $\phi(\linear_i)$.
\end{proposition}
\begin{proof}
  This is a corollary of Prop. \ref{powersum}. Indeed, one has
  \begin{equation}
    1+L=\sum_{n\geq 0} \theta(S_n)\quad
    \text{and}\quad E=\sum_{n\geq 1} n\theta( S_n). 
  \end{equation}
  Therefore
  \begin{equation}
    B=\sum_{n \geq 1} \theta(\Psi_n).
  \end{equation}
\end{proof}

We need to introduce the following notations.

The leaves of a planar binary tree with $n$ vertices are labelled from
$0$ to $n$ from left to right. The leaves with labels different from
$0$ and $n$ are called \textbf{inner leaves}. A \textbf{descent} in a
planar binary tree $t$ is the label of an $\backslash$-oriented inner
leaf. The descent set $D(t)$ of $t$ is the set of its descents.

The number of descents of a planar binary tree $t$ will be denoted $d(t)$. It
satisfies $0\leq d(t) \leq n-1$ for a tree $t$ of degree $n$. 

The \textbf{major index} $\maj(t)$ of $t$ is the sum of its descents.
For example, Fig. \ref{descente} displays a planar binary tree with
descent set $\{2,4\}$ and major index $2+4=6$.

\begin{figure}
\begin{center}
\includegraphics[scale=0.35]{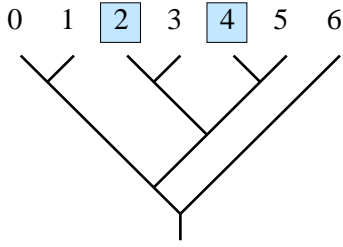} 
\caption{A planar binary tree $t$ with descents at $2$ and $4$.}
\label{descente}
\end{center}
\end{figure}

Let us recall that the descent set $D(I)$ corresponding to a composition
$I=(i_1,\dots,i_k)$ of $n$ is the set $\{i_1,i_1+i_2,\dots,i_1+\dots+i_k\}$.

\begin{proposition}
  \label{image_RI}
  The image by $\theta$ of $R_I$ is the sum
  \begin{equation}
    \sum_{\dessus{\# t=n}{ D(t)=D(I)}} t
  \end{equation}
  of all planar binary trees with $n$ vertices and descent set $D(I)$.
\end{proposition}
\begin{proof}
  This is a well-known property of the injection of $\nsym$ in $\dend$.
\end{proof}

In \cite{ncsf2}, elements $\Psi_n(\frac{A}{1-q})$, for $n\geq 1$, are
defined by some ``change of alphabet'' applied to the elements
$\Psi_n$. According to the proof of \cite[Theorem 6.11]{ncsf2}, they
are characterized by
\begin{equation}
 \label{eqPsiAq}
  \sum_{n\geq 1} \Psi_n(\frac{A}{1-q})=\left(\sum_{n\geq 0}
    S_n\right)^{-1} \left(\sum_{n\geq 1} q^n \Psi_n(\frac{A}{1-q})\right)\left(\sum_{n\geq 0} S_n\right)+\sum_{n\geq 1} \Psi_n.
\end{equation}

There is a classical isomorphism $\alpha$ from $\nsym$ to the direct
sum of all descent algebras of symmetric groups. By this morphism, up
to a multiplicative constant, each $\Psi_n(\frac{A}{1-q})$ is mapped
to a Lie idempotent with coefficients in $\QQ(q)$ in the descent
algebra of the $n^{th}$ symmetric group.

We can now state the precise relation between $\Omega_q$ and these Lie
idempotents.
\begin{proposition}
  \label{prop_psi_omega}
  The image of $(1-q)\Psi(\frac{A}{1-q})$ by $\theta$ is $\phi(\su{\Omega_q})$.
\end{proposition}
\begin{proof}
  Indeed, by Proposition \ref{eq_phi_omega}, one has
  \begin{equation}
    \sum_{n\geq 1}\phi(\su{\Omega_q}) =(1+L)^{-1} *(
    \phi(\su{\Omega_q})[q]) *(1+L) +(1-q) \sum_{n\geq 1} \phi(\linear_n).
  \end{equation}
  Then using Prop. \ref{theta_psi} and Eq. (\ref{eqPsiAq}), one gets that
  $\theta((1-q)\Psi(\frac{A}{1-q}))$ and $\phi(\su{\Omega_q})$ satisfy
  the same equation, hence they are equal.
\end{proof}

Let $\Omega_{q,n}$ be the homogeneous
component of degree $n$ of $\Omega_q$.
\begin{proposition}
  \label{omega_dend}
  One has  
  \begin{equation}
    \phi(\Omega_{q,n})=\frac{(-1)^{n-1}}{[n]_q} \sum_{\# t=n} (-1)^{d(t)}
    \qbinom{n-1}{d(t)}^{-1} q^{\maj(t)-\binom{d(t)+1}{2}} \,t.
  \end{equation}
\end{proposition}

\begin{proof}
  The Theorem 6.11 of \cite{ncsf2} tells that the element
  $(1-q)\Psi_n(\frac{A}{1-q})$ is
  \begin{equation}
    \frac{1}{[n]_q} \sum_{|I|=n} (-1)^{d(I)}
    \qbinom{n-1}{d(I)}^{-1} q^{\maj(I)-\binom{d(I)+1}{2}} \, R_I.    
  \end{equation}
  By Prop. \ref{prop_psi_omega}, the image by $\theta$ of this formula is
  $(-1)^{n-1} \phi(\Omega_{q,n})$. By Prop. \ref{image_RI}, this becomes the
  expected formula.
\end{proof}

\section{Arithmetic properties}

In this section, we obtain some properties of the denominators in
$\Omega_q$ and consider what happens when $q$ is specialized to $1,0$
and $\infty$.

\subsection{$q=1$}

Let us first note that the morphism $\phi$ from $\PLh$ to the
completed free dendriform algebra $\dendh$ is defined over $\QQ$ and
injective. Hence one can deduce results on $\Omega_q$ from results on
its image by $\phi$.

\begin{proposition}
  The series $\Omega_q$ is regular at $q=1$ and $\Omega_{q=1}=\Omega$.
\end{proposition}

\begin{proof}
  By Proposition \ref{omega_dend}, the image $\phi(\Omega_q)$ is
  regular at $q=1$, as $q$-binomial coefficients become usual binomial
  coefficients when $q=1$. Therefore $\Omega_q$ itself is regular at
  $q=1$.

  At $q=1$, the equation (\ref{quant_eq}) becomes the equation
  (\ref{class_eq3}). By uniqueness in Proposition \ref{uniqueness2},
  the value of $\Omega_q$ at $q=1$ is $\Omega$.
\end{proof}

Remark: knowing that $\Omega_{q=1}=\Omega$, one can use equation
(\ref{quant_recursion}) to compute simultaneously $\Omega_q$ and
$\Omega$ up to order $n$ in a $O(n^3)$ number of pre-Lie operations.

\smallskip

There is a lot of cancellations in the coefficients of $\Omega_q$,
leading to a reduced complexity of the denominators. Note that the
expected denominator of $\Omega_{q,n}$ (from recursion
(\ref{quant_recursion})) is the product $\prod_{d=2}^{n}(q^d-1)$. Let
$\Phi_d$ be the $d^{th}$ cyclotomic polynomial.

\begin{proposition}
  The common denominator of the coefficients of the element
  $\Omega_{q,n}$ divides the product $\prod_{d=2}^{n} \Phi_d$.
\end{proposition}

\begin{proof}
  For the image of $\Omega_q$ by $\phi$, this follows from Prop.
  \ref{omega_dend} and a simple property of the $q$-binomial
  coefficients: their only roots are simple roots at roots of unity,
  see \cite[Prop 2.2]{Guo_zeng}. This implies the same result for
  $\Omega_q$.
\end{proof}

\subsection{$q=0$}

Let us consider now what happens when $q=0$. Then $\Omega_0$ is
well-defined, $\Omega_q[q]$
vanishes and the equation (\ref{quant_eq}) becomes
\begin{equation}
  \Omega_0 = \pun - \pun \pl \Omega_0.
\end{equation}
It follows that $\Omega_0$ is the alternating sum of linear trees.

\subsection{$q=\infty$}

Let us now consider what happens when $q=\infty$. Let $\omega_{q,T}$
be the coefficient of the rooted tree $T$ in the expansion of
$\Omega_q$. We will call \textbf{valuation} at $q=\infty$ the smallest
exponent in the formal Laurent expansion in powers of $q^{-1}$ of an
element of $\QQ(q)$.

\begin{proposition}
  The valuation of $\omega_{q,T}$ at $q=\infty$ is at least $\# T-1$.
\end{proposition}

\begin{proof}
  This will follow from the recursion (\ref{quant_recursion}). This is
  true in degree $n=1$. Let us assume that $n\geq 2$. Then the
  valuation of $\pun \pl \Omega_{q,n-1}$ is at least $n-2$ by
  induction and the valuation of each term of the rightmost sum in
  equation (\ref{quant_recursion}) is at least $-1$. Hence the
  valuation of $\Omega_{q,n}$ is at least $n-1$.
\end{proof}

Hence there exists a limit $\Omega_\infty$ for $\Omega_q[q]/q$ when $q$ goes to
$\infty$ and the limit of $\Omega_q/q$ is zero.

The equation (\ref{quant_eq}), divided by $q$, becomes at $q=\infty$,
\begin{equation}
   \Omega_\infty  \pl \exp(\Omega) = \pun.
\end{equation}

By right action by $\exp(-\Omega)$, this is equivalent to
\begin{equation}
   \Omega_\infty = \pun \pl \exp(-\Omega).
\end{equation}

The element $\exp(-\Omega)$ is the inverse of $\exp(\Omega)$ in
$\UPLh$. This has been computed in \cite[\S 6.4]{ChapLive2}. More
precisely, the inverse of $\sum_{n \geq 1} \frac{1}{(n-1)!}\corol_{n}$
in the group of characters of the Connes-Kreimer Hopf algebra was
shown there to be
\begin{equation}
  \sum_T \frac{(-1)^{\#T-1}}{\aut(T)} T,
\end{equation}
where $\aut(T)$ is the cardinal of the automorphism group of the
rooted tree $T$. But it is known \cite{ChapLive1} that this group of
characters is isomorphic to the group of group-like elements in
$\UPLh$. Going through the isomorphism, one gets the following result.

\begin{proposition}
  The series $\Omega_\infty$ is given by
  \begin{equation}
     \Omega_\infty=\sum_T \frac{(-1)^{\#T-1}}{\aut(T)}T.
  \end{equation}
\end{proposition}

\section{Morphisms and images}

In this section, we consider two quotients of the free pre-Lie algebra
$\PL$ and the images of $\Omega_q$ in these quotients. We will use
some results of \cite{Chap_2002}.

\subsection{Morphism to the free associative algebra}

Consider the free (non-unital) associative algebra on one generator
$x$, denoted by $\QQ[x]_+$. As the associative product is also a
pre-Lie product, there exists a unique morphism of pre-Lie algebras
from $\PL$ to $\QQ[x]_+$ sending $\pun$ to $x$. This extends uniquely
to a morphism from $\PLh$ to the algebra $\QQ[[x]]_+$ of formal power
series in $x$ without constant term.

One can show that this morphism send the linear trees $\linear_n$ with
$n$ vertices to the monomials $x^n$ and all others trees to $0$.

It is known (see \cite{Chap_2002}) that the image of $\Omega$ is the formal power series
\begin{equation}
  \log(1+x)=\sum_{n \geq 1} \frac{(-1)^{n-1}}{n} x^n.
\end{equation}
Therefore the image of $\exp(\Omega)-1$ is just $x$. Note that the
right action is mapped to the product.

One deduces from Eq. (\ref{quant_eq}) that the image of $\Omega_q$ is
 the $q$-logarithm defined by
\begin{equation}
 \log_q(x)=\sum_{n\geq 1} \frac{(-1)^{n-1}}{[n]_q}x^n,
\end{equation}
which is the unique solution to the functional equation
\begin{equation}
  x \log_q(qx)=x \log_q(x)+(q-1)\,x-\log_q(qx)+\log_q(x).
\end{equation}

\subsection{Morphism for corollas}

As shown in \cite{Chap_2002}, the subspace of $\PL$ spanned by trees
that are not corollas is a two-sided pre-Lie ideal.

The quotient pre-Lie algebra is isomorphic to the following pre-Lie
algebra. Let us identify the image of the corolla $\corol_{n+1}$ with
$n$ leaves to $x^n$ for all $n\geq 0$. In particular, the tree $\pun$
is mapped to $1$. The underlying vector space is therefore identified
with $\QQ[x]$ and the pre-Lie product is
\begin{equation}
  x^p \pl x^q=
  \begin{cases}
    x^{p+1} \quad \text{if}q=0,\\
    0    \quad\text{else.}
  \end{cases}
\end{equation}

It is known (see \cite{Chap_2002}) that the image of $\Omega$ is the
generating function $\frac{x}{\exp(x)-1}$ for the Bernoulli numbers.

One can also show (using the description of the quotient $\pl$ product
given above) that the right action of the image of $\exp(\Omega)-1$ is
just given by the product by $\exp(x)-1$ and the right action by the
image of $\Omega_q$ is given by the product by $x$.

Then, from Eq. (\ref{quant_eq}), one gets that the image $F_q(x)$ of
$\Omega_q$ satisfies the following equation
\begin{equation}
  (\exp(x)-1)[q F_q(qx)]=x+q-1-q F_q(qx)+F_q(x).
\end{equation}

This functional equation is known (see for instance \cite{satoh}) to
describe the generating function
\begin{equation}
  F=\sum_{n \geq 0} \beta_n(q) \frac{x^n}{n!},
\end{equation}
where $\beta_q(n)$ are the $q$-Bernoulli numbers introduced by Carlitz
in 1948, see \cite{Carlitz_1948,Carlitz_1954,Carlitz_1958}.

Therefore the coefficients of the corollas in $\Omega_q$ are the
$q$-Bernoulli numbers of Carlitz.

\subsection{Morphism to a pre-Lie algebra of vector fields}

There exists an interesting morphism from $\PL$ to a pre-Lie algebra
of vector fields. We describe it here only as a side remark, as the
image of $\Omega_q$ seems to have no special property.

Consider the the vector space $V=\QQ[x]_+$, endowed with the following
pre-Lie product:
\begin{equation}
  (f \pl g)=x f \, \partial_x g.
\end{equation}
Then there is a unique morphism from $\PL$ to $V$ sending $\pun$ to
$x$.

This map has the following nice property: the coefficient of $x^n$ in
the image of a series $A$ is the sum of the coefficients of the trees
in the homogeneous component $A_n$ of $A$. The proof is just a check
that this sum-of-coefficients map defines a morphism of pre-Lie
algebra from $\PL$ to $V$.

\section{First terms of some expansions}

\begin{multline}
\Omega=\arb{0}-\frac{1}{2}\arb{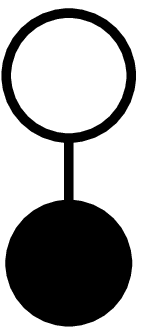}+\frac{1}{3} 
\arb{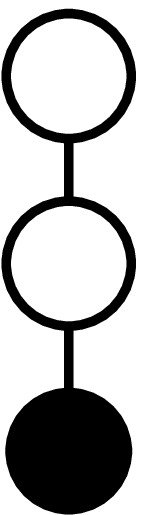}+\frac{1}{12}\arb{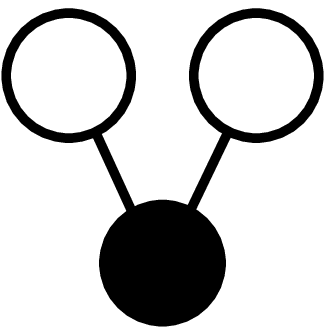} -\frac{1}{4}
\arb{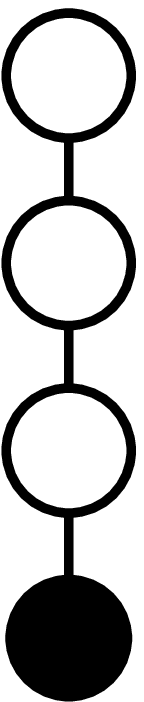}-\frac{1}{12}\arb{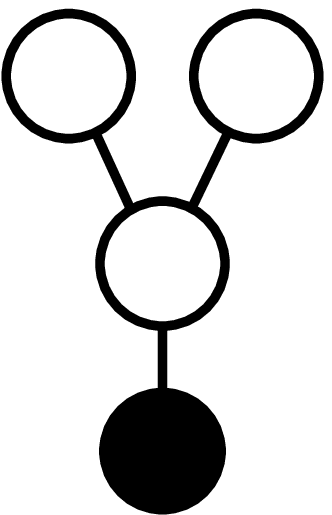}-\frac{1}{12}\arb{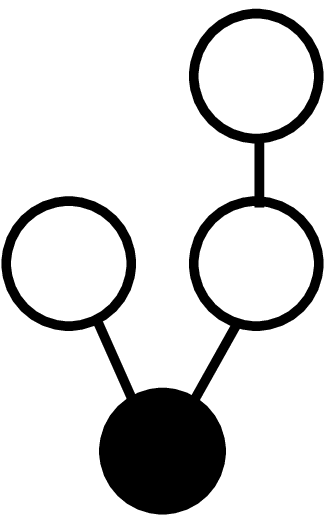} 
+\\
\frac{1}{5}\arb{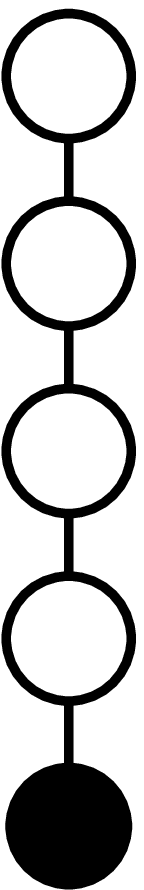}+\frac{3}{40}\arb{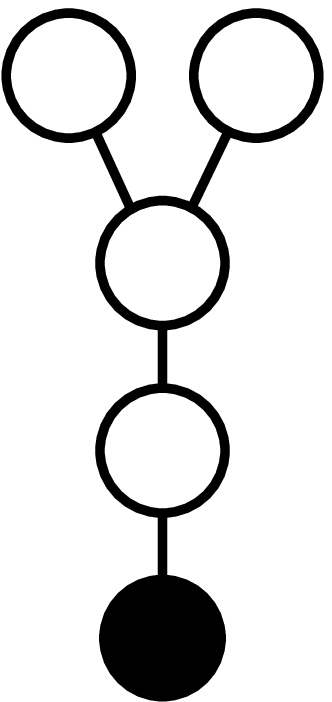}+\frac{1}{10}\arb{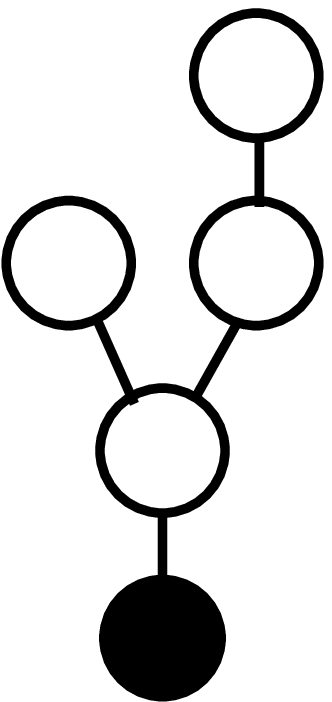}+\frac{1}{180}\arb{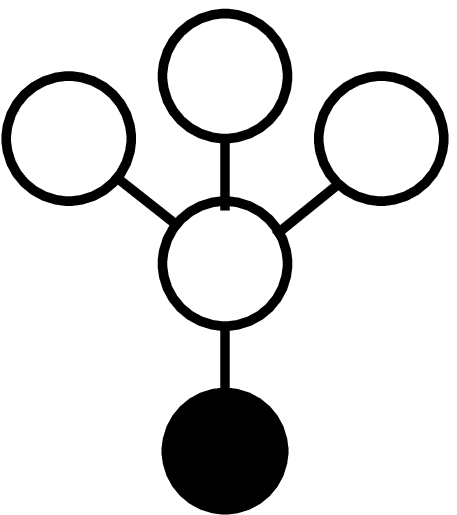}+\frac{1}{60}\arb{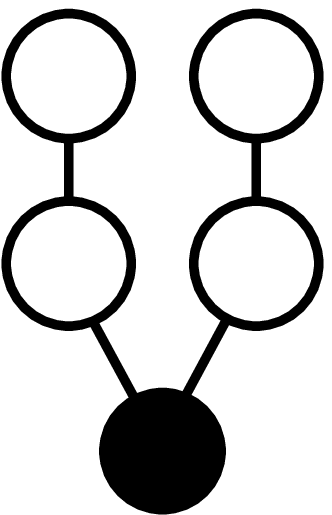}+\frac{1}{20}\arb{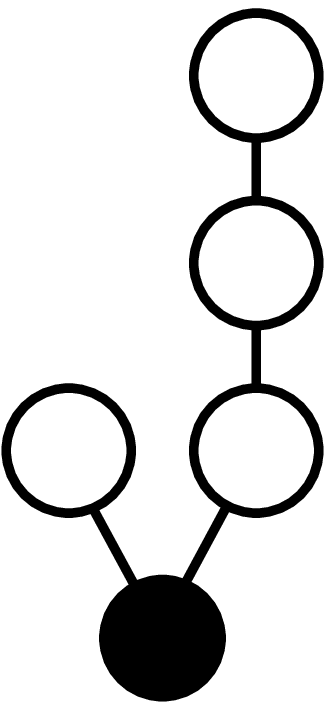}+\frac{1}{120}\arb{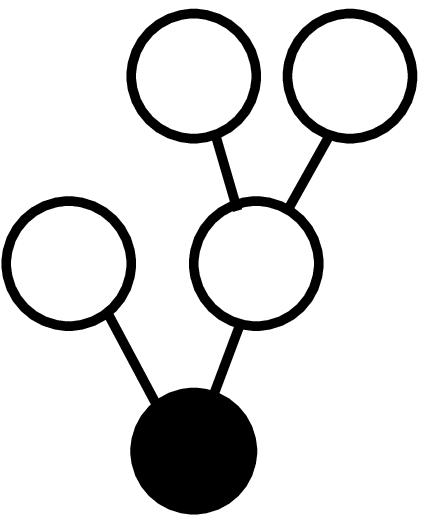}-\frac{1}{120}\arb{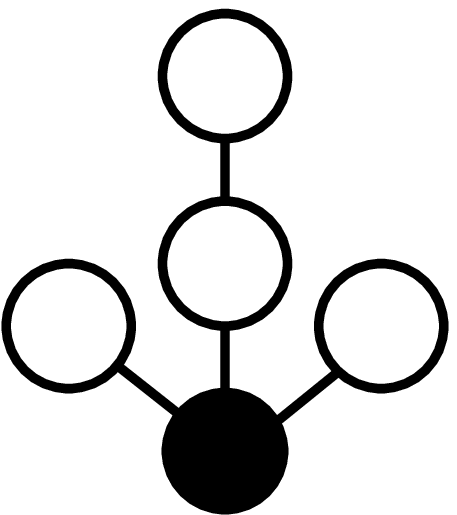}-\frac{1}{720}\arb{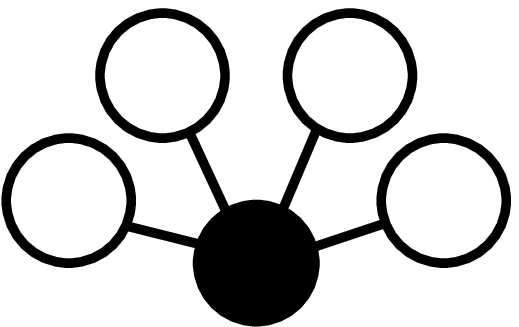}
+\,\cdots
\end{multline}

For all $n\geq1$, let $\Phi_n$ be the $n^{th}$ cyclotomic polynomial.

\begin{multline}
\Omega_q=\arb{0}-\frac{1}{\Phi_2}\arb{10}+\frac{1}{\Phi_3} 
\arb{110}+\frac{q}{2 \,\Phi_2 \Phi_3}\arb{200} \\
-\frac{1}{\Phi_2 \Phi_4}
\arb{1110}-\frac{q}{2\,\Phi_3 \Phi_4}\arb{1200}-\frac{q^2}{\Phi_2 \Phi_3
\Phi_4}\arb{2100} -\frac{q(q-1)}{6\, \Phi_2\Phi_3\Phi_4}\arb{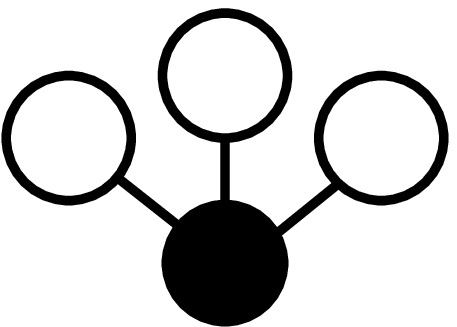}
+\\
\frac{1}{\Phi_5}\arb{11110}+\frac{q(1+q+q^2)}{2\,\Phi_2\Phi_4\Phi_5}\arb{11200}+\frac{q^2}{\Phi_4\Phi_5}\arb{12100}+\frac{q(q^3+q^2-1)}{6\,\Phi_3\Phi_4\Phi_5}\arb{13000}+\frac{q^4}{2\,\Phi_3\Phi_4\Phi_5}\arb{21010}+\frac{q^3}{\Phi_2\Phi_4\Phi_5}\arb{21100}+\\
\frac{q^2(q^3+q^2-1)}{2\,\Phi_2\Phi_3\Phi_4\Phi_5}\arb{22000}+\frac{q^2(q^3-q-1)}{2\,\Phi_2\Phi_3\Phi_4\Phi_5}\arb{31000}+\frac{q(q^4-q^3-2q^2-q+1)}{24\,
\Phi_2 \Phi_3\Phi_4\Phi_5}\arb{40000}
+\,\cdots
\end{multline} 

\begin{multline}
\Omega_\infty=\arb{0}-\arb{10}+ 
\arb{110}+\frac{1}{2}\arb{200}-
\arb{1110}-\frac{1}{2}\arb{1200}-\arb{2100}-\frac{1}{6}\arb{3000} 
+\\
\arb{11110}+\frac{1}{2}\arb{11200}+\arb{12100}+\frac{1}{6}\arb{13000}+\frac{1}{2}\arb{21010}+\arb{21100}+\\
\frac{1}{2}\arb{22000}+\frac{1}{2}\arb{31000}+\frac{1}{24}\arb{40000}
+\,\cdots
\end{multline}

\begin{multline}
 \Omega_0=\arb{0}-\arb{10}+\arb{110}-\arb{1110}+\arb{11110}
 +\,\cdots
 \end{multline}

\bibliographystyle{alpha}
\bibliography{qomega}

\end{document}